\documentclass[10pt,a4paper]{amsart}
\usepackage[margin=20mm]{geometry}
\usepackage[numbers]{natbib}
\usepackage[colorlinks]{hyperref}

\usepackage{amssymb,amsmath}
\usepackage{amsthm}
\usepackage{bbm}
\usepackage{stmaryrd}
\usepackage[usenames,dvipsnames]{xcolor}
\usepackage{color}

\input xy
\xyoption{all} 

\numberwithin{equation}{section}

\newtheorem{theorem}{Theorem}[section]
\newtheorem{corollary}{Corollary}[theorem]
\newtheorem{lemma}[theorem]{Lemma}
\newtheorem{proposition}[theorem]{Proposition}

\theoremstyle{definition}
\newtheorem{definition}[theorem]{Definition}
\newtheorem{remark}[theorem]{Remark}
\newtheorem{example}[theorem]{Example}

\newcommand{\Dleft}{[\hspace{-1.5pt}[}
\newcommand{\Dright}{]\hspace{-1.5pt}]}
\newcommand{\SN}[1]{\Dleft #1 \Dright}

\newcommand{\Id}{\mathbbmss{1}}

\newcommand{\rmi}{ \textnormal{i}}
\newcommand{\rmd}{\textnormal{d}}
\newcommand{\rme}{\textnormal{e}}

\newcommand{\N}{\mathbb{N}}

\DeclareMathOperator{\Der}{Der}

\font\black=cmbx10 \font\sblack=cmbx7 \font\ssblack=cmbx5 \font\blackital=cmmib10  \skewchar\blackital='177
\font\sblackital=cmmib7 \skewchar\sblackital='177 \font\ssblackital=cmmib5 \skewchar\ssblackital='177
\font\sanss=cmss10 \font\ssanss=cmss8 
\font\sssanss=cmss8 scaled 600 \font\blackboard=msbm10 \font\sblackboard=msbm7 \font\ssblackboard=msbm5
\font\caligr=eusm10 \font\scaligr=eusm7 \font\sscaligr=eusm5  \font\fraktur=eufm10
\font\sfraktur=eufm7 \font\ssfraktur=eufm5 
\font\bsymb=cmsy10 scaled\magstep2
\def\all#1{\setbox0=\hbox{\lower1.5pt\hbox{\bsymb
       \char"38}}\setbox1=\hbox{$_{#1}$} \box0\lower2pt\box1\;}
\def\exi#1{\setbox0=\hbox{\lower1.5pt\hbox{\bsymb \char"39}}
       \setbox1=\hbox{$_{#1}$} \box0\lower2pt\box1\;}

\def\tx#1{{\fam0\relax#1}}

\newfam\bifam
\textfont\bifam=\blackital \scriptfont\bifam=\sblackital \scriptscriptfont\bifam=\ssblackital

\newfam\blfam
\textfont\blfam=\black \scriptfont\blfam=\sblack \scriptscriptfont\blfam=\ssblack

\newfam\bbfam
\textfont\bbfam=\blackboard \scriptfont\bbfam=\sblackboard \scriptscriptfont\bbfam=\ssblackboard

\newfam\ssfam
\textfont\ssfam=\sanss \scriptfont\ssfam=\ssanss \scriptscriptfont\ssfam=\sssanss
\def\sss#1{{\fam\ssfam\relax#1}}

\newfam\clfam
\textfont\clfam=\caligr \scriptfont\clfam=\scaligr \scriptscriptfont\clfam=\sscaligr

\newfam\frfam
\textfont\frfam=\fraktur \scriptfont\frfam=\sfraktur \scriptscriptfont\frfam=\ssfraktur

\def\hpb#1{\setbox0=\hbox{${#1}$}
    \copy0 \kern-\wd0 \kern.2pt \box0}
\def\vpb#1{\setbox0=\hbox{${#1}$}
    \copy0 \kern-\wd0 \raise.08pt \box0}

\def\pmb#1{\setbox0\hbox{${#1}$} \copy0 \kern-\wd0 \kern.2pt \box0}
\def\pmbb#1{\setbox0\hbox{${#1}$} \copy0 \kern-\wd0
      \kern.2pt \copy0 \kern-\wd0 \kern.2pt \box0}
\def\pmbbb#1{\setbox0\hbox{${#1}$} \copy0 \kern-\wd0
      \kern.2pt \copy0 \kern-\wd0 \kern.2pt
    \copy0 \kern-\wd0 \kern.2pt \box0}
\def\pmxb#1{\setbox0\hbox{${#1}$} \copy0 \kern-\wd0
      \kern.2pt \copy0 \kern-\wd0 \kern.2pt
      \copy0 \kern-\wd0 \kern.2pt \copy0 \kern-\wd0 \kern.2pt \box0}
\def\pmxbb#1{\setbox0\hbox{${#1}$} \copy0 \kern-\wd0 \kern.2pt
      \copy0 \kern-\wd0 \kern.2pt
      \copy0 \kern-\wd0 \kern.2pt \copy0 \kern-\wd0 \kern.2pt
      \copy0 \kern-\wd0 \kern.2pt \box0}

\mathchardef\za="710B  
\mathchardef\zb="710C  
\mathchardef\zg="710D  
\mathchardef\zd="710E  
\mathchardef\zve="710F 
\mathchardef\zz="7110  
\mathchardef\zh="7111  
\mathchardef\zvy="7112 
\mathchardef\zi="7113  
\mathchardef\zk="7114  
\mathchardef\zl="7115  
\mathchardef\zm="7116  
\mathchardef\zn="7117  
\mathchardef\zx="7118  
\mathchardef\zp="7119  
\mathchardef\zr="711A  
\mathchardef\zs="711B  
\mathchardef\zt="711C  
\mathchardef\zu="711D  
\mathchardef\zvf="711E 
\mathchardef\zq="711F  
\mathchardef\zc="7120  
\mathchardef\zw="7121  
\mathchardef\ze="7122  
\mathchardef\zy="7123  
\mathchardef\zf="7124  
\mathchardef\zvr="7125 
\mathchardef\zvs="7126 
\mathchardef\zf="7127  
\mathchardef\zG="7000  
\mathchardef\zD="7001  
\mathchardef\zY="7002  
\mathchardef\zL="7003  
\mathchardef\zX="7004  
\mathchardef\zP="7005  
\mathchardef\zS="7006  
\mathchardef\zU="7007  
\mathchardef\zF="7008  
\mathchardef\zW="700A  
\mathchardef\zC="7009  

\newcommand{\be}{\begin{equation}}
\newcommand{\ee}{\end{equation}}

\newcommand{\bea}{\begin{eqnarray}}
\newcommand{\eea}{\end{eqnarray}}
\def\*{{\textstyle *}}
\newcommand{\R}{{\mathbb R}}

\newcommand{\C}{{\mathbb C}}
\newcommand{\Z}{{\mathbb Z}}

\newcommand{\s}{{\textstyle *}}

\newcommand{\A}{{\mathcal A}}

\def\Sec{\sss{Sec}}

\def\xi{\tx{i}}

\def\s*{{\scriptstyle *}}

\newcommand{\beas}{\begin{eqnarray*}}
\newcommand{\eeas}{\end{eqnarray*}}

\def\half{\frac{1}{2}}

\newcommand{\B}{\mathcal{B}}

\title{Almost commutative $Q$-algebras and derived brackets}

   \author{Andrew James Bruce}
   \address{Mathematics Research Unit, University of Luxembourg, Maison du Nombre   			6, avenue de la Fonte,  L-4364 Esch-sur-Alzette, Luxemburg}
   \email{andrewjamesbruce@googlemail.com}

\date{\today}
 
\begin{document}

\begin{abstract}
We introduce the notion of \emph{almost commutative Q-algebras} and demonstrate how the derived bracket formalism of Kosmann-Schwarzbach generalises to this setting. In particular, we construct `almost commutative Lie algebroids' following Va\u{\i}ntrob's Q-manifold understanding of classical Lie algebroids. We show that the basic tenets of the theory of Lie algebroids carry over verbatim to the almost commutative world. \par
\smallskip\noindent
\begin{tabular}{ll}
{\bf Keywords:} & noncommutative geometry;~almost commutative algebras;~Lie algebroids;~Q-manifolds\\
{\bf MSC 2010:}&\hspace{-8pt} \begin{tabular}[t]{ll}
                 \emph{Primary}:& 81R60 ;~46L87;~17B75 \\
                 \emph{Secondary}:& 53D17;~58A50  
                \end{tabular}
\end{tabular}
\end{abstract}

 \maketitle

\setcounter{tocdepth}{2}
 \tableofcontents

\section{Introduction}
When quantum effects of the gravitational field become relevant, it is expected that the geometry of space-time will depart from its classical nature.  From rather general arguments, it is likely that space-time at some level will no longer be a Riemannian manifold, but instead, take on some `noncommutative manifold'  structure. This could, for example, regulate the curvature singularities encountered in classical general relativity.  It is fair to say that noncommutative geometry is `work in progress' and unfortunately nature offers almost no clues as to exactly what we should expect of noncommutative geometry. We must also mention the impact of noncommutative geometry on condensed matter physics where, as just one example, great insight into the  integer quantum Hall effect has been gained \cite{Bellissard:1994}. From a mathematical perspective, noncommutative geometry is the natural progression of space-algebra duality and sheds light on what one could possibly mean by `geometry'.  There are many approaches to noncommutative geometry and here we broach the subject by employing a very mild form of noncommutativity.  \par 

Almost commutative algebras, also known as $\rho$-commutative algebras, have been studied since the 1980s following the work of Rittenberg \& Wyler \cite{Rittenberg:1978} and Scheunert \cite{Scheunert:1979} on generalised Lie algebras. Loosely, we have an algebra that is commutative up to some ``commutation factor''.  It was shown by Bongaarts \& Pijls \cite{Bongaarts:1994} that  almost commutative algebras  offer a convenient framework to develop a very workable form of noncommutative differential geometry. A natural example of such an algebra is the $\Z_2$-commutative algebra of global functions on a supermanifold where the commutation factor is simply a plus or minus sign. We remark that the possibility of algebras with more general gradings that just $\Z$ or $\Z_2$ was suggested in 1970 by Berezin and Kac \cite{Berezin:1970}.  In particular, almost commutative algebras are `close enough' to commutative and supercommutative algebras to allow many of the constructions found in differential geometry to be directly generalised. Using almost commutative algebras one can largely mimic classical differential geometry following the derivation based approach to noncommutative geometry  of Dubois-Violette \cite{Dubois-Violette:1988}.  Many of the subtleties of working with ``fully noncommutative'' algebras disappear when using almost commutative algebras. For example, $\rho$-derivations of an almost commutative algebra form a module over the whole of the algebra.    Generically one can develop lots of geometry with minimal fuss by employing almost commutative algebras, see for example \cite{Bongaarts:1994,Ciupala:2005,Ciupala:2005b,deGorsac:2012,Ngakeu:2007,Ngakeu:2012,Ngakeu:2017}. Specifically, one does not need to reach for the $C^*$-algebraic approach to noncommutative geometry as pioneered by Connes in order to build ``mildly'' noncommutative spaces.  Moreover, interesting and well-known examples of such algebras exist such  quantum hyperplanes, super-versions thereof, noncommutative tori,  and the quaternionic algebra. All the aforementioned algebras are of direct interest in physics from various perspectives. Noncommutative tori are one of the fundamental examples of a noncommutative space and they have served as a `testbed' or `sandbox' for the development of noncommutative geometry.   \par 
In this paper, we modify the understanding of Lie algebroids due to Va\u{\i}ntrob \cite{Vaintrob:1997} to the setting of almost commutative algebras. We will work with very particular almost commutative algebras that carry an additional $\N$-grading, which we will refer to as \emph{weight}. The assignment of weight is quite independent of the commutation rules. We view these algebras as the natural generalisation of a \emph{graded bundle} or a \emph{non-negatively graded manifold}, which are themselves generalisations of vector bundles. The original idea that an $\N$-grading can be used as a replacement of a linear/vector bundle structure is due to Voronov  \cite{Voronov:2002}. This idea has since been refined and applied by many authors (see for example \cite{Grabowski:2012,Roytenberg:2002,Severa:2005,Severa:2017,Voronov:2012}). These algebras will come equipped with a homological derivation (to be defined carefully in the main sections), and we will show that the derived bracket formalism of Kosmann-Schwarzbach \cite{Kosmann-Schwarzbach:1996,Kosmann-Schwarzbach:2004} (also see Voronov \cite{Voronov:2005}) allows one to construct a `shifted' $\rho$-Lie bracket on the space of particular derivations of the almost commutative algebra.  All these constructions parallel the classical understanding of Lie algebroids in terms of graded supermanifolds equipped with a homological vector field.  Indeed, up to a shift in Grassmann party, classical Lie algebroids are examples of the structures we consider. We will assume the reader has some familiarity with the various descriptions of classical Lie algebroids and in particular  Va\u{\i}ntrob's  approach via supermanifolds.\par 
While the notion of \emph{Lie--Rinehart pair} as an algebraic generalisation of a Lie algebroid has been well studied, our motivation for this work was to understand if the picture of Lie algebroids due to Va\u{\i}ntrob has a natural generalisation in the noncommutative world. While we offer no perspective on the full noncommutative world, the almost commutative world seems very tractable and we can rather directly mimic the `super-setting' of Lie algebroids. We remark that Lie algebroids, which  were first introduced by  Pradines \cite{Pradines:1967}, are prevalent throughout differential geometry, geometric mechanics, and geometric approaches to field theory in many guises. Our general reference for Lie algebroids is the book by Mackenzie \cite{Mackenzie:2005}. Moreover, understanding Lie algebroids and related objects in terms of supergeometry has slowly been gaining acceptance. It seems natural and fitting that a similar understanding  of `algebroids' in the noncommutative setting be reached. Furthermore, it is well known that the categories of Poisson manifolds and Lie algebroids are intimately intertwined. The relation between Poisson almost commutative algebras and particular graded version of Lie algebroids was explored by Ngakeu \cite{Ngakeu:2012,Ngakeu:2017}. Given that the quasi-classical limit of a `noncommutative space-time' is expected to be related to a Poisson algebra, understanding `noncommutative Lie algebroids' in all their formulations could be a fruitful avenue of exploration.

\medskip 

\noindent \textbf{Arrangement:} In Section \ref{Sec:AlmComAlg} we review the notion of an almost commutative algebra. It is here that we introduce the concept of homological $\rho$-derivations and almost commutative Q-algebras. In Section \ref{Sec:AlComNGAlg} we present the notion of an almost commutative non-negatively graded algebra.  In Section \ref{Sec:QAlgBrack} we study almost commutative Q-algebras equipped with a non-negative grading and show how the derived bracket formalism can be applied to this setting.  As a particular example, in Section \ref{Sec:DiffCalc} we re-examine first order differential calculi on almost commutative algebras and, just as in the classical case of the de Rham complex on a smooth manifold, we are able to interpret the $\rho$-Lie bracket on $\rho$-derivations as a derived bracket. Modules over almost commutative Q-algebras are the subject of Section \ref{Sec:Qmodules}. Here we give the adjoint and coadjoint modules as explicit examples. The adjoint module is fundamental in understanding deformations and symmetries of almost commutative Q-algebras as discussed in Section \ref{Sec:DefSymm}.

\section{Almost Commutative Algebras}\label{Sec:AlmComAlg}
We draw heavily on the work of Bongaarts \& Pijls \cite{Bongaarts:1994}, Ciupal\u{a} \cite{Ciupala:2005,Ciupala:2005b} and Ngakeu \cite{Ngakeu:2012,Ngakeu:2017} in this section. We will present no proofs for the statements in this section as they follow directly from calculations or can be found in the literature previously cited. Let $G$ be a discrete abelian group, which we write additively, examples of such groups are $\Z$, $\Z_2$, $\Z_p$ ($p$ prime) and finite direct sums thereof, e.g., $\Z^2 := \Z \times \Z$. Let $\A$ be a $G$-graded algebra. This means that, as a vector space, $\A = \oplus_{a \in G}\A_a$, and that the multiplication respects this grading in the sense that $\A_a \cdot \A_b \subset \A_{a+b}$ ($a,b \in G$). The $G$-degree of a homogeneous element  $f \in \A$, we denote as $|f|$. Now assume that we have a map $\rho : G \times G \rightarrow \mathbb{K}$, where we will take the ground field to be $\R$ or $\C$, that satisfies the following:
\begin{align*}
& \rho(a,b) = \rho(b,a)^{-1},\\
& \rho(a+b, c) = \rho(a,c)\rho(b,c),
\end{align*}
for all $a,b$ and $c\in G$.  One can deduce that these conditions imply that $\rho(a,b) \neq 0$, $\rho(0,b) = 1$ and that $\rho(c,c) = \pm 1$, for all for all $a,b$ and $c\in G$. Moreover, we also have that $\rho(a, b+c) = \rho(a,b)\rho(a,c)$.  The mapping $\rho$ is a particular $2$-cocycle on the group $G$ and is often referred to a \emph{commutation factor}. For the rest of this paper we will simply refer to a cocyle.  \par 
Two cocycles $\rho$ and $\rho\prime$ on $G$ are said to be \emph{equivalent cocycles} if and only of there exists an element $\phi \in \textrm{Aut}(G)$, i.e., an element of the automorphism group of $G$, such that $\rho\prime(a,b) = \rho(\phi(a), \phi(b))$ for all $a$ and $b\in G$.
\begin{example}
If $G = \Z$, then the trivial cocycle is given by $\rho(a,b) =1$ and the standard cocycle is given by $\rho(a,b) = (-1)^{ab}$.  The standard cocycle is, of course, just the standard sign factor.
\end{example}
\begin{example}
If $G = \Z\times \Z$ and $\mathbb{K} = \C$, then $\rho((n,m), (n'm')) = \rme^{2 \pi \rmi (nm' {-} mn')}$  is a cocycle.
\end{example}
As a $G$-vector space, we have the natural definition of \emph{$G$-parity} for $G$-homogeneous elements of $\A$:
$$\# f = \frac{1 - \rho(|f|, |f|)}{2}.$$
We then say that a homogeneous element $f \in \A$ is \emph{even} if $\# f = 0$ and \emph{odd} if $\# f = 1$. For homogeneous elements $f$ and $g \in \A$, we  define the \emph{$\rho$-commutator} as
$$[f,g]_\rho :=  fg \: {-} \:  \rho(|f|,|g|) \:g f\,. $$
Extension to non-homogeneous elements is via linearity. Whenever we write an expression for homogeneous elements, we will understand that extension to non-homogeneous elements is via linearity, although we will not state this explicitly.  The reader can directly verify the following properties:
\begin{align*}
& [\A_a, \A_b ]_\rho \subset \A_{a+b},\\
& [f,g]_\rho = {-} \rho(|f|, |g|) [g,f]_\rho.
\end{align*}
A direct calculation yields the \emph{$\rho$-Jacobi identity}
$$\rho(|f|, |h|)^{-1} [f, [g,h]_\rho]_\rho  + \rho(|h|, |g|)^{-1} [h, [f,g]_\rho]_\rho  + \rho(|g|, |f|)^{-1} [g, [h,f]_\rho]_\rho  =0.$$
The properties of the $\rho$-commutator naturally suggest the definition of a \emph{$\rho$-Lie algebra}, a notion which can be traced at least to Rittenberg \& Wyler \cite{Rittenberg:1978}, Scheunert \cite{Scheunert:1979} and Mosolova \cite{Mosolova:1981}. Such generalised Lie algebras are also known as \emph{Lie colour algebras}. We will later require the notion of a `shifted' or `anti' Lie algebra.  We allow the bracket to carry a non-trivial $G$-degree, however, we will insist on `oddness'. The reader should keep in mind the classical Schouten--Nijenuis bracket on the space of multivector fields on a manifold.  The physics motivated reader should keep in mind the antibracket in the BV-BRST formalism of gauge theory. This will help explain some of our conventions and nomenclature.
\begin{definition}\label{def:rhoAntiLieAlg}
Let $\mathfrak{g}$ be a $G$-graded vector space and let $\rho$ be a cocycle. Then $\mathfrak{g}$ is a \emph{$\rho$-Loday-Leibniz antialgebra of $G$-degree $|\mathfrak{g}|$} if there exists a bilinear map, which we refer to as a \emph{$\rho$-antibracket}
\begin{align*}
\mathfrak{g} \times \mathfrak{g} & \longrightarrow \mathfrak{g}\\
(f,g) & \mapsto [f,g]_\mathfrak{g},  
\end{align*}
such that 
\begin{enumerate}
\item $\rho(|\mathfrak{g}|,|\mathfrak{g}| ) = {-}1$;
\item $|[f,g]_\mathfrak{g}| = |\mathfrak{g}| + |f| + |g|$;
\item $[f, [g,h]_\mathfrak{g}]_\mathfrak{g} =  [[f,g]_\mathfrak{g}, h]_\mathfrak{g}+ \rho(|f| + |\mathfrak{g}|  , |g|+ |\mathfrak{g}| )  \: [g, [f,h]_\mathfrak{g}]_\mathfrak{g},$
\end{enumerate}
for all $f,g$ and $h \in \mathfrak{g}$.\par 
If in addition to the above bracket has the following $\rho$-skewsymmetry property
\begin{enumerate}
\setcounter{enumi}{3}
\item $[f,g]_\mathfrak{g} = {-} \rho(|f| + |\mathfrak{g}| ,|g| + |\mathfrak{g}| ) \: [g,f]_\mathfrak{g}$,
\end{enumerate}
then we speak of a \emph{$\rho$-Lie antialgebra of $G$-degree $|\mathfrak{g}|$}.
\end{definition}
In the above definition, the $\rho$-Jacobi identity we have written in so-called Loday--Leibniz form (cf. Loday \cite{Loday:1993}). This definition has the natural interpretation of $[f, -]_\mathfrak{g}$ being a $\rho$-derivation of $G$-degree $|\mathfrak{g}| + |f| $ of the `bracket algebra'. Note that this interpretation is quite independent of the $\rho$-skewsymmetry. \par 
\begin{remark}
Our nomenclature `Lie antialgebra' and `antibracket' has been hijacked from the BV-BRST formalism in physics. Mathematically, `anti' signifies a shift in the grading as compared to  the standard case of Lie superalgebras. In  the current context, `anti' signifies that  $\rho(|\mathfrak{g}|,|\mathfrak{g}| ) = {-}1$ and not $1$.
\end{remark}
With these notions in place, we have the following definitions that will be central to the rest of this paper.
\begin{definition}
Let $\A$ be a $G$-graded algebra equipped with a given cocycle $\rho$. Then $\A$ is said to be a \emph{$\rho$-commutative algebra} or an \emph{almost commutative algebra} if and only if $[f,g]_\rho = 0$ for all $f$ and $g \in \A$.
\end{definition}
Note that in the above we fix a cocycle in the definition of an almost commutative algebra. That is, we will not consider  group automorphisms and equivalent cocycles in any of the constructions in this paper. Moreover, associativity and  ``almost commutativity''  force the cocycle to have the properties postulated - this is carefully discussed in \cite{Covolo:2016b}. 
\begin{remark}
The term \emph{$\epsilon$-graded commutative algebras} was used in \cite{deGorsac:2012} and the term \emph{colour commutative algebras} was used in \cite{Lychagin:1995} for what we call  almost commutative algebras.
\end{remark}
In particular, for odd elements of an almost commutative algebra, it is clear that $f^2 =0$. Thus, we have nilpotent elements to contend with. However, the assignment of $G$-parity does \emph{not} determine the commutation rule $fg = \rho(|f|, |g|) \:  g f$.
\begin{definition}
 Let $\A$ and $\A'$ be almost commutative algebras  with respect to some fixed abelian group $G$ and cocycle $\rho$. Then a \emph{morphism of almost commutative algebras} is a $\mathbb{K}$-linear map $\phi^\# : \A' \longrightarrow A$, such that 
 \begin{enumerate}
 \item $|\phi^\#(f)| = |f|$, and
 \item $\phi^\#(fg) = \phi^\#(f) \: \phi^\#(g)$,
 \end{enumerate}
 for all $f$ and $g \in \A'$. 
\end{definition}
Note that we consider only degree zero morphisms, that is, morphisms do not change the $G$-degree of elements. This requirement is very natural from the perspective of constructing geometries, but from an algebraic perspective, this can be relaxed. The reader can easily verify that $\rho$-commutativity is preserved under such morphisms. Evidently, we obtain the category of $(G,\rho)$ almost commutative algebras.
\begin{example}[Commutative algebras]
Any commutative algebra can be considered as an almost commutative algebra by taking $G$ as the trivial group and the trivial cocycle. As a specific example, the algebra of global functions on a smooth manifold is a commutative algebra.
\end{example}
\begin{example}[Supercommutative algebras]
Any supercommutative algebra can be considered as an almost commutative algebra by taking $G = \Z_2$ and the cocycle to be $\rho(a,b) = (-1)^{a b}$. As a specific example, the algebra of global functions of a supermanifold is a supercommutative algebra.
\end{example}
\begin{example}[Quaternions \cite{Morier-Genoud:2010}]\label{exp:Quarternions}
 The quaternionic algebra $\mathbb{H}$ is generated by the elements $\Id, e_1, e_2$ and $e_3$, subject to the relations
\begin{align*}
e_i^2 & = - \Id, & e_ie_j = {-} e_j e_i ~  ~ (i \neq j) && e_1 e_2 = e_3, && e_3 e_1 = e_2, && e_2 e_3 = e_1.
\end{align*}
The quaternionic algebra is an almost commutative algebra with $G = \Z_2^2 = \Z_2 \times \Z_2$, where the generators are assigned the grading
\begin{align*}
|\Id| = (0,0), && |e_1| = (1,0), && |e_2| = (0,1), && |e_3| = (1,1).
\end{align*}
The relevant cocycle is given by
$$ \rho((a_1, b_1), (a_2, b_2)) = (-1)^{a_1 b_2 {-} a_2 b_1}\,.$$
\end{example}
\begin{example}[$\Z_2^n$-commutative algebras \cite{Berezin:1970,Morier-Genoud:2010}]
A $\Z_2^n$-graded, $\Z_2^n$-commutative algebra ($n \in \N$) can be considered as an almost commutative algebra by taking $G = \Z_2^n := \Z_2 \times \Z_2 \times, \cdots , \times \Z_2$  ($n$-times) and the cocycle to be $\rho(|f|, |g|) = (-1)^{\langle|f|, |g| \rangle}$, where $\langle - , - \rangle $ is the standard scalar product on $\Z_2^n$. In fact, provided we have a finite number of generators, any sign rule for the permutation of elements can be  accommodated  in this way (\cite[Theorem 2]{Morier-Genoud:2010}).   As a specific example, the algebra of global functions of a $\Z_2^n$-manifold is a $\Z_2^n$-commutative algebra (see Covolo et al. \cite{Covolo:2016}).
\end{example}
\begin{remark}
It has been long known that $\Z_2$-gradings play a fundamental r\^{o}le in physics as soon as fermions are taken into account. Far less well known is that  $\Z_2^n$-gradings  $(n>1)$ naturally appear in the context of Green's parastatistics, see for example Dr\"{u}hl, Haag \& ~Roberts \cite{Druhl:1970}.
\end{remark}

\begin{example}[2-d quantum plane]
The quantum plane is described by the algebra $S_2^q$, which is generated by the unit element and two linearly independent generators $x$ and $y$, subject to the relation 
$$xy  \: {-} \:   q\: yx =0,$$
where $q \in \R$  is fixed and non-zero. The algebra $S_2^q$ is naturally $\Z^2$-graded,
$$S_2^q = \bigoplus_{n,m}^\infty  \big (S_2^q  \big)_{n,m} \: ,$$
where $\big (S_2^q  \big)_{n,m}$ is the one-dimensional vector space spanned by products of the form $x^n y^m$. Note that $ \big (S_2^q  \big)_{n,m} =0$ if either $n$ or $m$ are negative.  It is not difficult to see that the relation on the generators translates to
$$(x^n y^m) (x^{n'} y^{m'}) = q^{nm' - mn'} \:  (x^{n'} y^{m'}) (x^n y^m).$$
Clearly, we have an almost commutative algebra by defining the cocycle  as
$$\rho\big( (n,m) , (n', m')  \big) = q^{nm' - mn'}. $$
We just remark that the $N$-dimensional quantum hyperplane can similarly be defined and understood in the context of almost commutative algebras. Similarly, one can consider `super' versions of quantum hyperplanes following Manin \cite{Manin:1989}.  Such algebras are often referred to as \emph{quasi-commutative algebras}.
\end{example}
\begin{example}[Noncommutative 2-torus]\label{exp:NCTorus}
The algebra of functions on a noncommutative 2-torus denoted as $\A_\theta$, is the algebra generated by two unitary generators, $u$ and $v$, subject to the relation
$$uv \: {-} \:  \rme^{2 \pi \rmi \theta } \: vu =0,$$
with $\theta \in \R$. Much like the example of the 2-d quantum plane the algebra $\A_\theta$ is naturally $\Z^2$-graded
$$\A_\theta = \bigoplus_{n,m}^\infty  \big (\A_\theta  \big)_{n,m} \: ,$$
where $\big (\A_\theta  \big)_{n,m}$ is the one-dimensional vector space spanned by products of the form $u^n v^m$.  Again in almost the same way as the 2-d quantum plane, we naturally have an almost commutative algebra by defining the cocycle as 
$$\rho\big( (n,m) , (n', m')  \big) = \rme^{2 \pi \rmi (nm' - mn')}. $$
Noncommutative $N$-tori can similarly be defined and these are also further examples of almost commutative algebras. 
\end{example}
\begin{remark}
As is well-known, the algebra of functions on noncommutative tori  are, in fact,  $C^*$-algebras and so are prototypical examples of quantum spaces according to Connes. However, we will not make use of $C^*$-algebra structures  in the  proceeding constructions nor in any further examples.
\end{remark}
\begin{example}[Taft algebras] Similar to the previous examples, the Taft algebra $T(p)$ for any positive integer $p$ is generated by the unit element and two linearly independent generators $x$ and $y$, subject to the relations
\begin{align*}
x^p = 1, && y^p = 0, && xy = \zx yx,
\end{align*}
where $\zx$ is a primitive $p^\textnormal{th}$-root of unity.  Naturally, this algebra is $\Z_p^2 := \Z_p \times \Z_p$-graded and indeed an almost commutative algebra where the cocycle is given by 
$$\rho\big( (n,m) , (n', m')  \big) = \zx^{nm' - mn'}. $$
\end{example}
\begin{remark}
It is well-known that Taft algebras are examples of a Hopf algebra. Indeed, Taft algebras provide the first examples of neither commutative, nor cocommutative Hopf algebras. However, we will not make use of  Hopf algebra structures in the  proceeding constructions nor in any further examples.
\end{remark}
\begin{definition}
Let $\A$ be a $G$-graded algebra (not necessarily almost commutative). Then a \emph{$\rho$-derivation} $X$ of $\A$ of $G$-degree $|X|$ is a linear map $X: \A \rightarrow \A$ of $G$-degree $X$, that satisfies the \emph{$\rho$-derivation rule} 
$$ X(fg) =  (Xf)g + \rho(|X|, |f|) f (Xg),$$
for all homogeneous elements $f$ and $g \in \A$.
\end{definition}
The reader can directly verify that the $\rho$-commutator of two $\rho$-derivation is again a $\rho$-derivation. Thus, the vector space of all $\rho$-derivations of $\A$ forms a $\rho$-Lie algebra of $G$-degree $0$. We denote this  $\rho$-Lie algebra as $\rho\Der(\A)$.  Provided $\A$ is an almost commutative algebra, then the vector space of  $\rho$-derivations of $\A$ is  in fact a left $\A$-module, with the action of $\A$ on $\rho\Der(\A)$ being the obvious one, i.e.,
$$(fX)g = f(Xg),$$
where $X \in \rho\Der(\A) $ and $f,g \in \A$. The reader can also quickly verify that we have a \emph{$\rho$-Leibniz rule} 
\begin{equation}\label{eq:LeibnizRuleLie}
[X, f Y ]_\rho = X(f) Y + \rho(|X|, |f|) \: f \: [X,Y]_\rho,
\end{equation}
for any and all $X,Y \in \rho\Der(\A)$ and $f\in \A$.
\begin{definition}\label{def:foliation}
Let $\A$ be an almost commutative algebra. Then a \emph{foliation in} $\A$ is a subspace $F \subset \rho\Der(A)$ such that
\begin{enumerate}
\item $F$ is a left $A$-submodule, and
\item $F$ is a $\rho$-Lie subalgebra. 
\end{enumerate}
The pair $(\A, F)$ is referred to as a \emph{foliated almost commutative algebra}. 
\end{definition}
 \begin{remark}
 The notion of a foliated commutative algebra, as far as we know, goes back to Jan Kubarski in unpublished notes from the 6th International Algebraic Conference in Ukraine, July 1-7, 2007. This notion directly generalises to the `almost' case because we have a module structure. As is well known, for more general noncommutative algebras the space of derivations is a module over the centre the algebra and \emph{not} the whole of the algebra.
 \end{remark}
Because the bracket on $\rho$-derivations of $\A$ is the $\rho$-commutator, we have an almost commutative generalisation of a homological vector field, i.e., an odd vector field on a supermanifold that `squares to zero' (see \cite{Aleksandrov:1997}).
\begin{definition}\label{def:almostCommQ}
Let $\A$ be an almost commutative algebra.   A $\rho$-derivation $Q\in \rho\Der(\A)$ is said to be a \emph{homological $\rho$-derivation of $\A$} if and only if
\begin{enumerate}
\item $\# Q = 1$;
\item $[Q,Q]_\rho =0.  $
\end{enumerate}
\end{definition}

For notational ease, we set $Q^2  :=  Q\circ Q = \half [Q,Q]_\rho$ , which from the above definition, is equal to zero. The notion of closed and exact elements of $\A$ makes perfect sense. That is, $f \in \A$ is $Q$-closed if $Qf =0$, and $Q$-exact if there exists some $g \in \A$ such that $f= Qg$. Clearly, as $Q^2 =0$, $Q$-exact elements are automatically $Q$-closed. Thus, we can develop a cohomology theory, though we do not pursue that here. We will refer to the pair $(\A, Q)$ as an \emph{almost commutative $Q$-algebra} taking our nomenclature from Schwarz \cite{Schwarz:1999,Schwarz:2003} who studied a much more general notion than  which we propose.  It is clear that the pair $(\A, Q)$ is, in fact, a foliated almost commutative algebra, see Definition \ref{def:foliation}. A quick calculation shows that for any $f$ and $g \in \A$
$$[f\:Q, g\:Q]_\rho = \big( f \:Q(g) ~{-}~ \rho(|Q| + |f| , |Q| + |g|) ~ g\: Q(f) \big)Q,$$
and so we indeed have a $\rho$-Lie subalgebra of $(\rho\Der(\A), [-,-]_\rho)$.\par 
 Our attitude is that is that an almost commutative $Q$-algebra is a `mild' noncommutative generalisation of a Q-manifold. That is, a supermanifold equipped with a Grassmann odd vector field that squares to zero. Various classical structures in classical differential geometry can be encoded efficiently by employing Q-manifolds, usually with some additional grading and extra structure such as a symplectic form. Our motivating examples are classical Lie algebroids, though we mention that Courant algebroids can also be understood in such terms following Roytenberg \cite{Roytenberg:2002}. In physics, Q-manifolds are essential in the BV-BRST formalism of gauge theories, and in particular the AKSZ construction \cite{Aleksandrov:1997}.\par 
 As noted by  Bongaarts \& Pijls \cite{Bongaarts:1994} almost commutativity is `close enough' to commutativity to rather directly mimic the standard construction of the  Cartan calculus on a smooth manifold. Moreover, the almost commutativity determines everything else: there are no real choices to be made here. We will briefly recall the key elements of the constructions here directing the reader to the cited literature for further details.\par 
We will write the left action of $\rho\Der(\A)$ on $\A$ as $(X, f ) \mapsto \langle X | f\rangle$.  Similarly for $p$-linear maps $\alpha_p : \times^p (\rho\Der(\A)) \rightarrow A$, we write $(X_1, X_2, \cdots , X_p, \alpha_p)\mapsto \langle X_1 , X_2, \cdots, X_p| \alpha_p\rangle $. Differential $p$-forms are then defined in `almost' the same way as the classical case. First, the space of zero forms is defined as $\Omega^0(\A) := \A$. Then for $p\geq 1$, $\Omega^p(\A)$ are defined as the $G$-graded spaces of $\rho$-alternating $p$-linear maps (in the sense of $\A$-modules). To spell this structure out,  we have linearity
$$\langle f X_1, X_2, \cdots , X_p | \alpha_p \rangle = f \langle  X_1, X_2, \cdots , X_p | \alpha_p \rangle,$$
and the $\rho$-antisymmetry
$$\langle  X_1, X_2, \cdots , X_i , X_{i+1}, \cdots      ,X_p | \alpha_p \rangle =   {-}  \rho(|X_i| , |X_{i+1}|)  \: \langle  X_1, X_2, \cdots , X_{i+1},X_i, \cdots      ,X_p | \alpha_p \rangle.$$
In this was we obtain a $G$-graded right $\A$-module with 
$$|\alpha_p| = |\langle  X_1, X_2, \cdots , X_p | \alpha_p \rangle| - (|X_1| + |X_2| + \cdots + |X_p|),$$
and the right action being 
$$\langle  X_1, X_2, \cdots , X_p | \alpha_p  \: f \rangle = \langle  X_1, X_2, \cdots , X_p | \alpha_p \rangle \: f. $$
 Naturally, the direct sum $\Omega^\bullet(\A) = \bigoplus_{p =0}^\infty \Omega^p(\A)$ is a $G$-graded $\A$-module.\par 
 Just as in the classical case, we have the de Rham differential $\rmd$, and for any  $X \in \rho\Der(\A)$ we have the interior product $i_X$ and Lie derivative $L_X$ all acting on $\Omega^\bullet(\A)$. The definitions of these operators are `almost' the same as the classical definitions.\par 
 The \emph{de Rham differential} is the linear map $\rmd : \Omega^p(\A) \rightarrow \Omega^{p+1}(\A)$, defined as $\langle X| \rmd f\rangle  = X(f)$ for $f \in \Omega^{0}(\A) = \A$, and for $p \geq 1$ it is defined as
\begin{align*}
\langle X_1, X_2, \cdots , X_{p+1} | \rmd \alpha_p \rangle & := \sum_{j=1}^{p+1} (-1)^{j-1} \rho\big(\sum_{l=1}^{j-1}|X_l|, |X_j|  \big)X_j\big( \langle X_1, \cdots, \widehat{X_j}, \cdots, X_{p+1}  |  \alpha_p\rangle \big)\\
& + \sum_{l<j<k \leq p+1} (-1)^{j+k}\rho \big(\sum_{l=1}^{j-1}|X_l|, |X_j|  \big) \: \rho \big(\sum_{l=1}^{j-1}|X_l|, |X_k|  \big) \\
 & \times  \rho \big(\sum_{l=j+1}^{k+1}|X_l|, |X_k|  \big) \: \langle [X_j, X_k]_\rho, X_1, \cdots , \widehat{X_j}, \cdots , \widehat{X_k}, \cdots , X_{p+1} | \alpha_p\rangle. 
\end{align*}
 Given any  $X \in \rho\Der(\A)$, the \emph{interior derivative} is defined as
 $$\langle X_1, \cdots , X_{p-1} | i_X\alpha_p  \rangle := \rho\big( \sum_{l=1}^{p-1} |X_l| , |X| \big) \: \langle X, X_1, \cdots , X_{p-1} | \alpha_p  \rangle,  $$
 with $i_X f =0$, and the \emph{Lie derivative} is defined as 
\begin{align*}
\langle X_1, \cdots , X_p | L_X \alpha_p  \rangle & := \rho\big( \sum_{l=1}^p  |X_l|, |X|\big)  \: X\big( \langle X_1, \cdots , X_p | \alpha_p  \rangle  \big)  {-} \sum_{j=1}^p \rho\big( \sum_{l=j}^p  |X_l|, |X|\big)\\
 & \times \langle  X_1. \cdots [X, X_j]_\rho,\cdots, X_p | \alpha_p   \rangle .
\end{align*}
The space $\Omega^\bullet(\A)$ carries a natural algebraic structure in `almost' the same way as the differential forms on a manifold do. The most direct way to see this is to note that $\Omega^\bullet(\A)$ is  $G'$-graded, where $G' := \Z\times G$. Any ($G$-homogeneous) $p$-form has $G'$-degree $|\alpha_p|' := (p, |\alpha_p|)$. We then define the cocycle $\rho' : G'\times G' \rightarrow \R $ as $\rho'((p,a) , (q, b)) :=  (-1)^{pq} \rho(a,b)$. The space of differential forms $\Omega^\bullet(\A)$ is then a $\rho'$-commutative $G'$-algebra.\par 
Note that $\rmd$, $i_X$ and $L_X$ are $\rho'$-derivations on $\Omega^\bullet(\A)$ of $G'$-degrees $|\rmd|' = (1,0)$, $|i_X|' = (-1, |X|)$ and $|L_X|' =( 0, |X|)$. It can be shown that the \emph{Cartan identities} generalise to the current situation in the expected way:

\begin{align}\label{eqn:CartanCalc}
 & [\rmd, \rmd]_{\rho'} = 0 , &&  [\rmd, i_X]_{\rho'} = L_X, & [\rmd, L_X]_{\rho'}=0,\\
\nonumber  & [i_X, i_Y]_{\rho'} = 0, && [L_X, i_Y]_{\rho'} = i_{[X,Y]_\rho},& [L_X, L_Y]_{\rho' } = L_{[X,Y]_\rho}.
\end{align}
Observe that given any  $\rho$-commutative algebra $\A$, the de Rham complex $(\Omega^\bullet(\A) ,\rmd) $ is an almost commutative Q-algebra (see Definition \ref{def:almostCommQ}). This is clear as $\Omega^\bullet(\A)$ is a $\rho'$-commutative $G'$-algebra, $\# \rmd = 1$, and the first of the Cartan identities can be written as $\rmd^2 = \half [\rmd, \rmd]_{\rho'}=0$.
\begin{remark}
The notion of an almost Q-algebra and in particular the  example $(\Omega^\bullet(\A) ,\rmd) $ is more general than what one usually means by a first-order differential calculus (cf. Woronowicz \cite{Woronowicz:1987}).  Specifically, there is no general condition that everything is generated by one-forms that are of the form $\alpha = \sum_k\rmd g_k \: f_k$ for $f_k$ and $g_k \in \A$.
\end{remark}

\section{Almost Commutative Non-negatively Graded Algebras}\label{Sec:AlComNGAlg}
Graded manifolds and particularly non-negatively graded  manifolds provide an economic framework to encode various classical structures like Lie algebroids, Courant algebroids, Dirac structures and so on (see for example \cite{Roytenberg:2002,Severa:2005,Severa:2017,Voronov:2002}). In the current setting, will be interested in almost commutative algebras that have a compatible $\N$-grading, which we refer to as \emph{weight}. The basic idea is that we wish to algebraically capture the important features of the \emph{polynomial algebra} on a graded bundle (see \cite{Voronov:2002} and for a review of the theory of graded bundles may consult \cite{Bruce:2017}). \par
Consider an $(\N, G)$-bigraded algebra 
$$\A =  \bigoplus_{\stackrel{a \in G}{i \in \N}} \A_{(i, a)}.$$
As a bigraded algebra we have $\A_{(i,a)} \cdot \A_{(j ,b)} \subset \A_{(i+j,a +b)}$. We will denote the $G$-degree of a homogeneous element $f \in \A$ as before, i.e., $|f| \in G$. The weight we will denote by  $\widetilde{f}$. Let $\rho$ be a $G$-cocycle. We will insist that the commutation rules be determined by the $G$-degree (via $\rho$) and that the weight is independent of the $\rho$-commutative property. Their r\^{o}les in the geometry we develop will be quite different. This leads us to the following definition.
\begin{definition}
An $(\N, G)$-bigraded algebra $\A$ is said  to be an \emph{almost commutative non-negatively graded algebra} if and only if it is a $\rho$-commutative algebra. 
\end{definition}
\begin{remark}
Naturally, we could allow a $\Z$-grading instead of just an $\N$-grading. In physics, an example of such a grading would be `ghost number'. The important thing here is that the $\N$ or $\Z$-grading plays a very different r\^{o}le to the $\Z_2$-grading in the BV-BRST formalism. Likewise for  almost commutative non-negatively graded algebra the $\N$-grading and $G$-grading have very different meanings. Roughly, the $\N$-grading will encode a `graded bundle' structure, while the $G$-grading encodes the noncommutativity of the geometry.
\end{remark}
We need to understand what it means for the `weight to be bounded'. In the classical setting of non-negatively graded manifolds this means that the weight of the local coordinates is bounded from above. We mimic this in the following way. First, we define 
$$\A_i :=  \bigoplus_{a \in G} \A_{(i,a)}.$$
Then we define $\textnormal{Alg}\big( \A_i \big)$ as the $\rho$-commutative algebra generated by elements in $\A_0 \oplus \A_1 \oplus \cdots \oplus \A_i$.  Naturally, this is itself an $(\N, G)$-bigraded algebra and an almost commutative  subalgebra of $\A$. We can then consider the following set
$$\A_{\langle i \rangle}:=  \textnormal{Alg}\big( \A_{i-1} \big) \cap \A_i.$$
In words, this is the set of weight $i$ elements of $\A$ that are built from elements of weight less that $i$. We then make the following definition.
\begin{definition}
An almost commutative non-negatively graded algebra $\A$ is said to be of \emph{degree} $n \in \N$ if and only if
$$\A_i\slash \A_{\langle i \rangle} = \{ 0\},$$
for all $i > n$.
\end{definition}
As it stands, our definition of the degree means that if $\A$ is of degree $n$, then it is automatically also of degree $n+1$. However, by degree, we will be explicitly be referring to the \emph{minimal degree} from this point on, i.e, the lowest possible degree.\par 
The space of $\rho$-derivations of an almost commutative non-negatively graded algebra is  $\Z$-graded. That is, we naturally have $\rho$-derivations that carry a negative degree with respect to the weight of the algebra.
\begin{lemma}
Let $\A$ be an almost commutative non-negatively graded algebra of degree $n$. Then the space of $\rho$-derivations of $\A$ is bounded from below at $-n$, i.e., there are no non-zero $\rho$-derivations of $\Z$-degree less than $-n$.
\end{lemma}
 \begin{proof}
 Suppose $X$ is a $\rho$-derivation of $\A$ of $\Z$-degree $-i$ and that $i>n$. Clearly, such a derivation will annihilate any element of $\A$ of weight $< i$. The only possible non-trivial action is on elements of weight $>i$. However, as the algebra we are considering is of degree $n$, elements of weight $>n$ are products of elements of weight $<n$. Thus, via the $\rho$-derivation rule, elements of weight $>n$ are also annihilated by $X$. We conclude that the lowest possible degree of a non-zero $\rho$-derivation is $-n$.
 \end{proof} 
 We can naturally decompose the $\rho$-Lie algebra of $\rho$-derivation of $\A$ (of degree $n$) as the sum of two  $\rho$-Lie subalgebras
$$\rho\Der(\A) = \rho\Der_{-}(\A) \oplus \rho\Der_{+}(\A), $$
 where
 \begin{align*}
 \rho\Der_{-}(\A) = \bigoplus^{i = -1}_{i = -n} \rho\Der_{i}(\A), &&
 && \rho\Der_{+}(\A) = \bigoplus_{i \geq 0} \rho\Der_{i}(\A)
 \end{align*}
 Some direct observations here are the following:
 \begin{enumerate}
 \item The  $\rho$-Lie algebra  $\rho\Der_{-}(\A)$ is nilpotent.
 \item  $\rho$-Lie algebra  $\rho\Der_{-n}(\A)$ is abelian.
 \end{enumerate}
 The situation here is exactly the same for vector fields on a non-negatively graded manifold, where the abelian group is simply $\Z_2$ and cocycle is the standard sign factor, see  Voronov \cite{Voronov:2010}.

\section{Almost Commutative $Q$-algebras and $\rho$-antibrackets}\label{Sec:QAlgBrack}
Taking Va\u{\i}ntrob \cite{Vaintrob:1997} as our cue, we make the following definition.
\begin{definition}\label{def:AlmostQdeg1}
Fix some abelian group $G$ and cocycle $\rho$. An \emph{almost commutative $Q$-algebra of degree $1$} is a pair $(\A, Q)$ consisting  of:
\begin{enumerate}
\item an almost commutative non-negatively graded algebra $\A$ of degree one;
\item a homological $\rho$-derivation of $\A$ of $\Z$-degree one.
\end{enumerate}
A \emph{morphism of almost commutative $Q$-algebras of degree $1$} is a bi-graded algebra morphism
$$\phi^\# : \A' \longrightarrow \A,$$
that relates the respective $\rho$-derivations, i.e.,
$$ \phi^\# \circ Q' = Q \circ \phi^\#.$$
\end{definition} 

The idea is that an almost commutative $Q$-algebra of degree $1$ should be a noncommutative version of a Lie algebroid  (or really a \emph{Lie antialgebroid}), or more carefully we have a noncommutative version of the Chevalley--Eilenberg complex associated with a Lie algebroid.  Our attitude is that this complex is the `true' fundamental definition of a Lie algebroid.
\begin{remark}
Schwarz \cite{Schwarz:1999,Schwarz:2003} defined a Q-algebra as a $\Z$ (or $\Z_2$) graded associative algebra $A$ equipped with a derivation of degree $1$ and  an element $\omega \in A$ of degree $2$ (the curvature) such that $Q^2(-) = [\omega , -]$. The algebras we define are closely related to Schwarz's Q-algebras when we replace the commutator with the $\rho$-commutator and realise that $[\omega, -]_\rho =0$. 
\end{remark}

We can modify the derived bracket construction of Kosmann-Schwarzbach \cite{Kosmann-Schwarzbach:2004} (also see Voronov \cite{Voronov:2005}) to the current setting of almost commutative algebras. For all of this section, we will specifically concentrate on the degree $1$ case without further explicit reference.\par 
We will change nomenclature slightly in order to fit with classical vector bundles. By the \emph{sections of $\A$}, we mean the $G$-graded vector space $\Sec(\A) := \rho\Der_{-}(\A)$. Recall that we can also consider this space as an abelian $\rho$-Lie algebra, this will be important in the constructions. Moreover,  $\Sec(\A)$ naturally has the structure of a left module over the $\rho$-commutative algebra $\B := \bigoplus_{a \in G} \A_{(0,a)}$. Geometrically we think of the algebra $\B$ as the algebra of function the `base almost commutative manifold' and $\A$ as the algebra of fibre-wise polynomials on the total space of  the `almost commutative vector bundle'. 
\begin{definition}
Let $(\A, Q)$ be an almost commutative $Q$-algebra of degree $1$. Then the \emph{derived $\rho$-antibracket} is the bilinear map
\begin{align*}
&\Sec(\A)\times \Sec(\A)  \longrightarrow \Sec(\A)\\
 & (\sigma, \psi)  \: \mapsto \SN{\sigma, \psi}_{Q} :=  \rho(|\sigma| +|Q|, |Q|)[[Q, \sigma]_\rho , \psi]_\rho.
\end{align*}
\end{definition} 
Let us now examine the properties of the derived $\rho$-antibracket.
\begin{theorem}\label{thm:rhoLieAlg}
Let $(\A, Q)$ be an almost commutative $Q$-algebra of degree $1$ and let $\SN{-,-}_Q$ denote the corresponding derived $\rho$-antibracket. The derived $\rho$-antibracket  exhibits the following properties:
\begin{enumerate}
\item $|\SN{\sigma, \psi}_Q| = |\sigma| + |\psi| + |Q| $;
\item $\SN{\sigma, \psi}_Q = {-} \rho(|\sigma| + |Q| , |\psi| + |Q|) \:  \SN{\psi ,\sigma}_Q$; 
\item $\SN{\sigma, \SN{\psi, \omega}_Q}_Q = \SN{\SN{\sigma,\psi}_Q , \omega}_Q + \rho(|\sigma| + |Q|, |\psi| + |Q|)\: \SN{\psi, \SN{\sigma, \omega}_Q}_Q$. 
\end{enumerate}
In short, $\Sec(\A)$, once equipped with the derived $\rho$-antibracket, is a $\rho$-Lie antialgebra (see Definition \ref{def:rhoAntiLieAlg}).
\end{theorem}
\begin{proof}These statements follow directly via calculation in almost exactly the same way as the classical $\Z$ or $\Z_2$-graded case. 
\begin{enumerate}
\item This follows directly from the definition and the properties of $G$-degree.
\item This follows from the $\rho$-Jacobi identity  and the $\rho$-skewsymmetry for the $\rho$-commutator.  Specifically, it is easy to see that the $\rho$-Jacobi identity together with the $\rho$-skewsymmetry  imply that
$$[Q,[\sigma, \psi]_\rho]_\rho = [[Q, \sigma]_\rho, \psi]_\rho + \rho(|\sigma|, |\psi|) \: [[Q,\psi], \sigma].$$
Because we are dealing with an abelian $\rho$-Lie subalgebra the left-hand side of the above vanishes and we can thus write
$$\SN{\sigma, \psi}_Q  = \rho(|\sigma| + |Q|, |Q|) \rho(|\sigma|, |\psi|) \rho(|Q| , |\psi| + |Q|) \SN{\psi, \sigma}.$$
Now using the properties of the cocycle $\rho$ we arrive at
$$\SN{\sigma, \psi}_Q  =  \rho(|Q|, |Q|)\rho(|\sigma| + |Q|,  |\psi|+|Q|)  \SN{\psi, \sigma}_Q,$$
and then as $\rho(|Q|, |Q|) =-1$, we obtain the required symmetry. 
\item Directly from the definitions we have
$$\SN{\sigma, \SN{\psi, \omega}_Q}_Q = \rho(|\sigma | + |\psi| , |Q|)\:  [[Q, \sigma]_\rho, [  [Q, \psi]_\rho ] , \omega ]_\rho;$$
we then make use of the $\rho$-Jacobi identity for the $\rho$-commutator to obtain
\begin{align*} &= && \rho(|\sigma| + |\psi|, |Q|) \: [[[Q,\sigma]_\rho , [Q,\psi]_\rho  ]_\rho, \omega  ]_\rho \\
&  && +\rho(|\sigma| + |\psi|, |Q|) \rho(|\sigma| + |Q| , |\psi| + |Q|)[[Q,\psi]_\rho,[  [Q,\sigma]_\rho ,\omega]_\rho   ]_\rho;
\end{align*}
and again using the $\rho$-Jacobi identity for the $\rho$-commutator we arrive at
\begin{align*} &= && \rho(|\sigma| + |\psi|, |Q|) \: [[[[Q,\sigma]_\rho,Q]_\rho, \psi]_\rho,\omega]_\rho \\
& && + \rho(|\sigma| + |\psi|, |Q|) \rho(|\sigma|+ |Q|, |Q|)\: [[Q,[[Q,\sigma]_\rho,\psi]_\rho]_\rho , \omega]_\rho \\
&  && +\rho(|\sigma| + |\psi|, |Q|) \rho(|\sigma| + |Q| , |\psi| + |Q|)[[Q,\psi]_\rho,[  [Q,\sigma]_\rho ,\omega]_\rho   ]_\rho.
\end{align*}
Using the definition of the derived $\rho$-antibracket and once more the $\rho$-Jacobi identity for the $\rho$-commutator we finally obtain
\begin{align*}
\SN{\sigma, \SN{\psi, \omega}_Q}_Q   =  & \SN{\SN{\sigma,\psi}_Q , \omega}_Q + \rho(|\sigma| + |Q|, |\psi| + |Q|)\: \SN{\psi, \SN{\sigma, \omega}_Q}_Q\\
 & {-} \frac{1}{2} \rho(|\sigma| + |\psi| + |Q|)\rho(|\sigma| , |Q|)\: [[[[Q,Q]_\rho, \sigma ]_\rho , \psi]_\rho, \omega ]_\rho.
\end{align*}
Then as $[Q,Q]_\rho =0$ we obtain the required result.
\end{enumerate}
Putting the above together, and using the fact that $\rho(|Q|,|Q|) ={-}1$ we see that we do indeed have the structure of a $\rho$-Lie antialgebra.
\end{proof}
We also have a module left module structure of $\Sec(\A)$ to take into account. Following the classical picture of Lie algebroids, we define an anchor map in the following way.
\begin{definition}
Let $(\A, Q)$ be an almost commutative $Q$-algebra of degree $1$. The associated \emph{anchor map} is the morphism of (left) $\B$-modules
$$\mathrm{a}_Q  : \Sec(\A) \longrightarrow \rho\Der(\B),$$
defined by
$$\mathrm{a}_Q(\sigma)f :=  \rho(|\sigma| + |Q| , |Q|) [Q, \sigma](f) = \sigma(Q(f)), $$
for all $\sigma \in \Sec(\A)$ and $f \in \B$.
\end{definition}
 To be clear, we have that 
 $$\mathrm{a}_Q(f \sigma) = f  \: \mathrm{a}_Q(\sigma),$$
for all $\sigma \in \Sec(\A)$ and $f \in \B$.\par 
A direct calculation establishes the following result.
\begin{proposition}
Let $(\A, Q)$ be an almost commutative $Q$-algebra of degree $1$. Furthermore let $\SN{-,-}_Q$ and $\mathrm{a}_Q$ denote the associated derived $\rho$-antibracket and anchor map, respectively. The derived $\rho$-antibracket satisfies the following $\rho$-Leibniz rule:
$$\SN{\sigma, f \: \psi}_Q = \mathrm{a}_Q(\sigma)f  +  \rho(|\sigma| + |Q|, |f| ) \: f \: \SN{\sigma, \psi}_Q.$$
 \end{proposition}
 For Lie algebroids, it is well known that the compatibility of the anchor map and the bracket follow from the Jacobi identity.  That is, the anchor map is a morphism of Lie algebras. The situation in the current context is essentially identical: the proof follows in exactly the same way as the classical case. For computational ease, we define the $\rho$-Jacobiator as the mapping
 $$J_Q: \Sec(\A) \times \Sec(\A) \times \Sec(\A) \longrightarrow \Sec(\A),$$
 given by
 $$J_Q(\sigma, \psi, \omega) = \SN{\sigma, \SN{ \psi, \omega}_Q }_Q  - \SN{\SN{\sigma, \psi}_Q, \omega}_Q  - \rho(|\sigma| + |Q|, |\psi| + |Q|) \SN{\psi, \SN{\sigma, \omega}_Q}_Q.$$
Naturally, the  $\rho$-Jacobiator vanishes in our case as $Q^2 =0$. 
\begin{proposition}\label{prop:ancbracketcompat}
Let $(\A, Q)$ be an almost commutative $Q$-algebra of degree $1$. Furthermore let $\SN{-,-}_Q$ and $\mathrm{a}_Q$ denote the associated derived $\rho$-antibracket and anchor map, respectively.  Then the derived $\rho$-antibracket and anchor map are compatible in the sense that
$$\mathrm{a}_Q\big( \SN{\sigma, \psi}_Q\big) = \big[ \mathrm{a}_Q(\sigma) , \mathrm{a}_Q(\psi) \big]_\rho,$$
for all $\sigma$ and $\psi \in \Sec(\A)$
\end{proposition}
\begin{proof}
After short computation,  we see that
\begin{align*}
& J_Q(\sigma , \psi, f \omega) && = \mathrm{a}_Q(\sigma) \big(\mathrm{a}_Q(\psi)f \big)\omega  {-} \rho(|\sigma| + |Q| , |\psi| + |Q|) \mathrm{a}_Q(\psi) \big(\mathrm{a}_Q(\sigma)f \big)\omega  {-} \big(\mathrm{a}_Q\big( \SN{\sigma, \psi}_Q\big)f \big) \omega\\  
& && + \rho(|\sigma|  + |\psi| + |Q|, |f|) \rho(|Q|, |f|) f \: J_Q(\sigma, \psi, \omega),
\end{align*}
for all $\sigma, \psi$ and $\omega \in \Sec(\A)$ and $f \in \B$. Now simply using the fact that the $\rho$-Jacobiator for the derived $\rho$-antibracket vanishes established the result.
\end{proof}
We see that in analogy with the classical theory of Lie algebroids, which provides a framework for foliations of manifolds,  almost commutative Q-algebras of degree $1$  provide a class of foliations in almost commutative algebras.
\begin{corollary}
The image of the anchor map $\textnormal{Im}\:\mathrm{a}_Q \subset \rho\Der(\B)$, is a foliation in $\B$, see Definition \ref{def:foliation}.
\end{corollary}
\begin{definition}
The \emph{characteristic foliation}  of  an almost commutative $Q$-algebra of degree $1$ is the foliated almost commutative algebra $(\B,\textnormal{Im}\:\mathrm{a}_Q )$.
\end{definition}

\noindent \textbf{Statement:} Given the initial data of an almost commutative Q-algebra of degree $1$, $(\A, Q)$, we canonically have the structure of a $\rho$-Lie antialgebra on the space of sections of $\A$ (Theorem \ref{thm:rhoLieAlg}), together with an anchor map that is compatible with the   derived $\rho$-antibracket (Proposition \ref{prop:ancbracketcompat}). Up to our conventions with grading, the resulting derived structures are an almost commutative version of a Lie algebroid structure (cf. Ngakeu \cite{Ngakeu:2017} who uses a different choice of grading).
\begin{example}[Trivial structures]
Any almost commutative non-negatively graded algebra of degree $1$ can be equipped with the trivial homological $\rho$-derivation, i.e., $Q=0$. Geometrically, we think of a vector bundle equipped with the zero bracket and zero anchor. Moreover, we can consider any $\rho$-commutative algebra $\B = \oplus \B_a$ to be an almost commutative non-negatively graded algebra of degree $1$ by declaring that $\A_{(i,a)} = \{ 0\}$, except when $i=0$, and simply define $\A_{(0,a)} = \B_{a}$. Geometrically, this corresponds to the zero vector bundle.  
\end{example}
\begin{example}[The de Rham complex]
Given any $\rho$-commutative algebra $\B = \oplus \B_a$, it is clear that the de Rham complex $(\Omega^\bullet(\B), \rmd)$ is an almost commutative Q-algebra of degree $1$, where we take $G' = \Z\times G$ and the cocycle to be one previously discussed, see  Section \ref{Sec:AlmComAlg}.
\end{example}
\begin{example}[The canonical action on the noncommutative $2$-torus]
Continuing example \ref{exp:NCTorus}, it is well-known that the module of $\rho$-derivations on the noncommutative $2$-torus $\A_\theta$ is spanned by
\begin{align*}
& \delta_u = 2 \pi \rmi u \frac{\partial}{\partial u}, &\textnormal{and}&& \delta_v = 2 \pi \rmi v \frac{\partial}{\partial v}.
\end{align*}
Moreover, these $\rho$-derivations are the infinitesimal generators of the canonical action of the classical $2$-torus $\mathbb{T}^2$ on the noncommutative $2$-torus. The infinitesimal action is specified by two complex numbers $(c_u, c_v)$.  Following the idea of BRST quantisation, we ``promote'' these complex number to ``ghosts'' and construct a ``BRST differential''. That is, we construct an almost commutative Q-algebra of degree $1$ by first defining $G' = \N \times \Z\times \Z$ and
\begin{align*}
& |u|' = (0,1,0),  & |v|' = (0,0,1), && |\eta_u|' = (1,0,0), && |\eta_v|' = (1,0,0),
\end{align*}
where the cocycle is the obvious one, i.e., $\rho'\big((l, m,n) , (l', m' n') \big) = (-1)^{l l'}\rho\big((m,n) , (m'n') \big)$.  The homological $\rho'$-derivation we define to be
$$Q= 2 \pi \rmi  \, \eta_u u \frac{\partial}{\partial u} \: + \:  2 \pi \rmi \,\eta_v v \frac{\partial}{\partial v},$$
which is clearly of $G'$-degree $(1,0,0)$. We thus interpret this construction as an almost commutative version of an action Lie algebroid. 
\end{example}

The  $G' = (\Z, G)$ degree of the de Rham differential suggests a natural class of  almost commutative Q-algebras.  In particular, consider the algebra generated by a finite number of generators $\{x^a, \zx^\alpha \}$, where we define the $G'$ degrees to be $|x^a|' := (0, |a|)$ and $|\zx^\alpha|' := (1, |\alpha|)$. The $\rho$-commutation laws are then
\begin{align*}
x^a x^b = \rho(|a|, |b|) \: x^b x^a, && \zx^\alpha x^a  = \rho(|\alpha|, |a|)\:  x^a \zx^\alpha, && \zx^\alpha \zx^\beta = {-} \rho(|\alpha|, |\beta|) \:  \zx^\alpha \zx^\beta.
\end{align*}
The algebra $\A$ be take to be the algebra generated by $\{x^a, \zx^\alpha\}$, with possible further relations on the $x's$, however, this will not affect what we write next. We understand the algebra to be a formal power series algebra in $\zx$.  Clearly, by defining $\widetilde{x}=0$ and $\widetilde{\zx} =1$, we obtain an almost commutative algebra of degree $1$. We can now write down the most general $(1,0)$-degree $\rho$-derivation:
$$Q = \zx^\alpha Q_\alpha^a(x) \frac{\partial}{\partial x^a} + \frac{1}{2}\zx^\alpha \zx^\beta Q_{\beta \alpha}^\gamma (x) \frac{\partial}{\partial \zx^\gamma},$$
where we have assumed that $|Q_\alpha^a| = |a| + |\alpha|$ and $|Q_{\beta \alpha}^\gamma| = |\alpha| + |\beta| + |\gamma|$. Note that we have the $\rho$-skewsymmetry $Q_{\beta, \alpha}^\gamma = {-} \rho(|\beta| , |\alpha|)\: Q_{\alpha \beta}^\gamma$. The reader should note that all of this is exactly as expected given the classical case (see Va\u{\i}ntrob \cite{Vaintrob:1997}).
Next we assume that the $\rho'$-derivation is homological, i.e., $Q^2 =0$. This imposes conditions on the components. After a direct, but not  illuminating calculation, one can derive the \emph{structure equations}:
\begin{align*}
& Q_\alpha^a \frac{\partial Q_\beta^b}{\partial x^a} ~{-}~ \rho(|\alpha|, |\beta|)\:  Q_\beta^a \frac{\partial Q_\alpha^b}{\partial x^a}~{-}~ Q_{\alpha \beta}^\gamma Q_\gamma^b =0,\\
& \sum_{\stackrel{\textnormal cyclic}{\alpha, \beta, \gamma} } \rho^{-1}(|\alpha|, |\beta|)\: \left(Q_\alpha^a \frac{\partial Q_{\beta \gamma}^\delta}{\partial x^a} ~{-}~ Q_{\alpha \beta}^\epsilon Q_{\epsilon \gamma}^\delta \right) =0.
\end{align*}
Again, these structure equations are formally identical to the classical structure equations for a Lie algebroid, but now we have to take into account extra factors due to the $\rho$-commutativity.

\section{First-order differential calculi and the Lie bracket}\label{Sec:DiffCalc}
We can construct a first-order differential calculi (in the sense of Woronowicz \cite{Woronowicz:1987}) over any almost commutative algebra $\A$ in the following way (due to Bongaarts \& Pijls \cite{Bongaarts:1994}). First, as standard, we define $\Omega_0^0(\A) := \A$, and then we define 
$$\Omega_0^1(\A) := \left\{ \sum_k \rmd f_k \: g_k  ~ | ~ f_k, g_k \in \A\right  \}, $$
where the differential is the de Rham differential defined earlier. The $\rho'$-commutative algebra of differential forms is then generated by $\A$ and  $\Omega_0^1(\A)$. In general, we have a subalgebra of $\Omega^\bullet(\A)$. Clearly in this was we can interpret $(\Omega^\bullet_0(\A), \rmd)$ as an almost commutative Q-algebra of degree $1$.
\begin{proposition}
Given any almost commutative algebra $\A$, there is a natural bijection 
$$\rho\Der(\A) \cong \Sec(\Omega^\bullet(\A)).$$
\end{proposition}
\begin{proof}
Given any $X \in \rho\Der(\A)$, we can send it to a unique $i_X \in \rho'\Der_(\Omega^\bullet(\A)) :=  \Sec(\Omega^\bullet(\A))$. Recall that part of the definition of the interior derivative is that $i_X(\rmd f) = X(f)$. \par 
In the other direction, consider an arbitrary $V \in \Sec(\Omega^\bullet(\A))$, of $G$-degree $|V|$. Just on weight grounds it is clear that $V(g) =0$ for any $g\in \A$. Thus, $V(\rmd f \: g) = V(\rmd f)\:g$.  In particular, this implies that $V\circ \rmd|_\A$ is a unique $\rho$-derivation on $\A$ of $G$-degree $V$. Direct calculation show we have the required derivation rule
$$V(\rmd(fg)) = V( \rmd f \: g + f \: \rmd  g) = V(\rmd f) \: g + \rho(|V|, |f|) \: f \: V(\rmd g).$$
Observing that
$$i_{V \circ \rmd |_\A}(\rmd f) = V(\rmd f) = V\circ \rmd|_\A(f) ,$$
shows we have constructed mutual inverses. Thus we have the desired bijection.
\end{proof}
The above proposition states that - not surprisingly - that we can identify sections of $\Omega^\bullet_0(\A)$ with $\rho$-derivations of $\A$. This is of course exactly what have in the classical case that $\A = C^\infty(M)$ where $M$ is a (finite dimensional) smooth manifold.  One has to keep in mind that by working with the algebras we are working in a dual picture to the geometry.  \par 
\noindent \textbf{Statement:} For any almost commutative algebra, the $\rho$-Lie bracket on $\rho\Der(\A)$ is  identified as the derived $\rho'$-antibracket with respect to the de Rham differential. That is, we make the identification (as sections of $\Omega^\bullet_0(\A)$) 
$$ \SN{i_X, i_Y}_\rmd := [[\rmd, i_X]_{\rho'}, i_Y  ]_{\rho'} = i_{[X,Y]_\rho} \, ,$$
using the Cartan identities (\ref{eqn:CartanCalc}). Note that this is in complete agreement with the classical case (see  Kosmann-Schwarzbach \cite{Kosmann-Schwarzbach:2004}).\par
Let us assume that the underlying almost commutative algebra $\A$ is finitely generated by $\{x^a \}$. We define $|x^a| := |a| \in G$ and  write the $\rho$-commutation rule as
$$x^a x^b = \rho(|a|, |b|) \: x^b x^a.$$
To construct the first-order differential calculi we append the (free) generators $\rmd x^b$ to the algebra and extend to a $G' = (\Z, G)$ grading as previously described. That is, we define the grading as
\begin{align*}
|x^a| = (0, |b|), && |\rmd x^b| = (1, |b|).
\end{align*}
The $\rho$-commutation rules are thus
\begin{align*}
x^a \rmd x^b  = \rho(|a|, |b|) \; \rmd x^b x^a, &&  \rmd x^a\rmd x^b =  {-} \rho(|a|, |b|) \; \rmd x^b  \rmd x^a.
\end{align*}
Any $\rho$-derivation can then be defined in terms of (algebraic) partial derivatives. Specifically, any $\rho$-derivation can be mapped to the interior derivative which is given by
$$i_X = X^a(x)\frac{\partial}{\partial \rmd x^a}.$$
The de Rham differential is given by
$$\rmd = \rmd x^a \frac{\partial}{\partial x^a}.$$
Clearly, we have an almost commutative Q-algebra of degree $1$. Following the definitions, we see that
$$\SN{i_X, i_Y}_\rmd = \left(X^a(x)\frac{Y^b(x)}{\partial x^a} ~ {-} ~ \rho(|X|, |Y| ) \:  Y^a(x)\frac{X^b(x)}{\partial x^a} \right) \frac{\partial}{\partial \rmd x^b}\,,$$
as fully expected.

\section{Q-modules over almost commutative Q-algebras}\label{Sec:Qmodules}
 Lie algebroid modules  \`{a} la Va\u{\i}ntrob are important  in the representation and deformation theory of Lie algebroids. Thus,  it makes sense to carefully define the related notion for almost commutative Q-algebras of degree $1$.
 \begin{definition}
Let $(\A, Q)$ be an almost commutative Q-algebra of degree $1$. A \emph{(left) Q-module over $\A$} is a   $(\Z, G)$-graded vector space $\mathcal{M}$, equipped with the following structures:
\begin{itemize}
\item An action
\begin{align*}
\mu :  \:  & \A \otimes_\mathbb{K} \mathcal{M} \longrightarrow \mathcal{M}\\
 & (a, v) \mapsto \mu(a,v) =: a v,
\end{align*}
that is $(\Z,G)$-degree preserving. By an action we mean that $\mu(a, \mu(b,v)) = \mu (ab, v)$.
\item A differential
$$\nabla :  \mathcal{M} \longrightarrow \mathcal{M}, $$
of $G$-degree $|Q|$ and $\Z$-degree $1$, that satisfies the following $\rho$-Leibniz  rule,
$$\nabla (a v) =  Q(a)\: v + \rho(|Q|, |a|) \:  a \: \nabla(v),$$
for all $a \in \A$ and $v \in \mathcal{M}$.  By a differential we naturally mean that $\nabla^2 =0$.
\end{itemize}
 \end{definition}
Naturally, it makes sense to refer to the map $\nabla$ as a \emph{flat connection}. Moreover, one can speak of the cohomology of an almost commutative Q-algebra of degree $1$ with values in the module $\mathcal{M}$. A right Q-module over $\A$ is similarly defined.
\begin{example}[Adjoint module]\label{exp:AdjointModule}
It is clear that $\rho\Der(\A)$ is a $\Z$-graded module over $\A$ by considering weight. We claim that $\nabla(-) := [Q, -]_\rho$ provides the required flat connection. The action is the standard one. The bi-degree of the defined connection is obvious and the flatness follows from $\nabla^2(-) = \half [[Q, Q] , -] =0$. The Leibniz rule in the definition of a flat connection is exactly the  $\rho$-Leibniz rule for the $\rho$-commutator, i.e., \eqref{eq:LeibnizRuleLie}.
\end{example}
\begin{example}[Coadjoint module]
Consider the module of differential forms on $\A$ which we bi-grade using  $G$-degree and the weight rather than `form degree';
$$\Omega^\bullet (\A) = \bigoplus_{\stackrel{i \geq 1}{a \in G}} \Omega^\bullet _{(i,a)}(\A).$$
Here it is convenient to consider the right module structure over $\A$. Then using the Cartan calculus, we claim that $\nabla := L_Q$ provides the structure of a flat connection. The condition in the bi-degree is clear and from the Cartan calculus we have that
$$[L_Q, L_Q]_{\rho'} = L_{[Q,Q]_\rho} =0.$$
Directly from the properties of the Lie derivative we obtain 
$$\nabla(\omega \: a) = \nabla( \omega) \: a  + \rho(|Q|, |\omega|) \:  \omega \: Q(a).$$
Thus we have the structure of a right Q-module.
\end{example}
\section{Deformations and symmetries of almost commutative Q-algebras}\label{Sec:DefSymm}
Again, we will stick to the degree $1$ case explicitly. Example \ref{exp:AdjointModule} shows that the space of $\rho$-derivations of an almost commutative algebra of degree $1$ canonically comes equipped with a flat connection. More than this, the $\rho$-Jacobi identity for the $\rho$-commutator shows that we have a derivation rule
\begin{equation}\label{eqn:nablederrule}
\nabla[X,Y]_\rho = [\nabla X, Y]_\rho ~+~ \rho(|Q|, |X|) \; [X, \nabla Y]_\rho \, .
\end{equation}
Thus, we have what one can call a \emph{differential $\rho$-Lie algebra}. The general ethos of deformation theory is that a differential Lie algebra always controls the deformation of some structure. The case at hand is no different.
\begin{theorem}
Deformations of an almost commutative Q-algebra of degree $1$ are controlled by the differential $\rho$-Lie algebra $\big(\rho\Der(\A), \, [-,-]_\rho,\, \nabla = [Q,-]_\rho \big)$. Specifically, $Q_X := Q + X$, $X \in \rho\Der(\A)$, is a homological $\rho$-derivation of $G$-degree $|Q|$ and weight $1$ if and only if
\begin{enumerate}
\item $|X| = |Q|$, $\widetilde{X} =1$, and
\item  $\nabla X + \frac{1}{2} [X, X]_\rho =0$.
\end{enumerate}
\end{theorem}
\begin{proof}
The condition on the $G$-degree and weight is obvious. The second part requires only a short calculation:
\begin{align*}
[Q_X,Q_X]_\rho &= ~ [Q,Q]_\rho + [Q, X]_\rho + [X, Q]_\rho +  [X,X]_\rho\\
               &= ~ 2 [Q , X]_\rho + [X,X]_\rho.
 \end{align*}
 Thus, if we insist on $[Q_X, Q_X]_\rho =0$, then we require that  $\nabla X + \frac{1}{2} [X, X]_\rho =0$.
\end{proof}
In other words, we require that $X$ satisfy the required Maurer--Cartan equation.
\begin{remark}
For the case of classical Lie algebroids, the reader can consult Ji \cite{Ji:2014}. Note that Crainic \& Moerdijk \cite{Crainic:2008} also define a differential graded Lie algebra controlling the deformation theory of Lie algebroids  and that this  is isomorphic to the differential graded Lie algebra studied by Ji. However, the use of Q-manifolds in this setting vastly simplifies the constructions. Historically, the definition of the adjoint module for Lie algebroids can be traced back to Va\u{\i}ntrob \cite{Vaintrob:1997} and  Grabowska, Grabowski \& P.~Urba\'{n}ski \cite{Grabowska:2003}. The deformation theory of VB-algebroids  has been explored by La Pastina \& Vitagliano \cite{LaPastina:2018}.

\end{remark}
\begin{definition}
Let $(\A,Q)$ be an almost commutative Q-algebra of degree $1$. Then a $\rho$-derivation $X \in \rho\Der_0(\A)$ is said to be an \emph{infinitesimal symmetry} or just a \emph{symmetry}, if and only if
$$ [X, Q]_\rho =0\,.$$
\end{definition}
Equivalently,  $X \in \rho\Der_0(\A)$ is a symmetry if and only if $\nabla X =0$. This follows trivially from $[X,Q]_\rho = {-}\rho(|X|, |Q|) \nabla X$. Thus, we will refer to a symmetry $X$ as an \emph{inner symmetry} if and only if there exists a section $\omega \in \Sec(\A)$ such that $X =\nabla \omega$.  In reverse, every section leads to an inner symmetry in an obvious way.
\begin{proposition}
The set of symmetries of $(\A, Q)$ is a $\rho$-Lie algebra with respect to the $\rho$-commutator bracket.
\end{proposition}
\begin{proof}
It is evident, due to the linearity of the $\rho$-commutator bracket, that constant multiples and sums of symmetries are again symmetries. Thus, we naturally have the structure of a $G$-graded vector space. The only thing to prove is that symmetries close under the $\rho$-commutator bracket. First, note that the $\rho$-commutator bracket is of weight zero, and so the weight is preserved. Second, we need to show that the $\rho$-commutator bracket of two symmetries is again a symmetry, but this is obvious in light of \eqref{eqn:nablederrule}.
\end{proof}
Note that we have insisted that symmetries be of weight zero. The reason for this is to ensure that  $\Sec(\A)$ is closed under the action of any symmetry. We have not made any assumption about the $G$-degree.
\begin{proposition}
Let $(\A, Q)$ be an almost commutative Q-algebra of degree $1$ and let $X\in \rho\Der_0(\A)$ be an infinitesimal symmetry. Then
$$L_X\SN{\sigma, \psi}_Q =  \SN{L_X\sigma, \psi}_Q ~+~ \rho(|X|, |Q| + |\sigma|) \: \SN{\sigma, L_X\psi}_Q,$$
for all $\sigma$ and $\psi \in \Sec(\A)$.
\end{proposition}
\begin{proof}
The proof follows from direct calculation using the $\rho$-Jacobi identity repeatedly. Explicitly,
\begin{eqnarray*}
L_X \SN{\sigma, \psi}_Q &=& \rho(|\sigma | + |Q| , |Q| ) \: [ X,[ [    Q,\sigma]_\rho\psi]_\rho  ]_\rho\\
                        &=& \rho(|\sigma | + |Q| , |Q| ) \: \big( ~[[[X,Q]_              							\rho, \sigma]_\rho, \psi]_\rho
                        +  \rho(|X|,|Q|)[ [Q,[X,\sigma]_\rho ]_\rho \psi]									_\rho \\
                        && +\rho(|X|, |Q| + |\sigma |) [[Q,\sigma]_\rho, 								[X,\psi]_\rho]_\rho ~ \big)\\
                        &=& \SN{L_X\sigma, \psi}_Q ~+~ \rho(|X|, |Q| + |								\sigma|) \: \SN{\sigma, L_X\psi}_Q ~+~ \rho(|								\sigma| +|Q| , |Q|) \: [[L_X Q, \sigma]_\rho  , 							\psi]_\rho.
\end{eqnarray*}
Then from the condition that we have a symmetry, we obtain the desired result.
\end{proof}
\begin{example}
Examining the de Rham complex $(\Omega^\bullet(\A), \rmd)$ of any almost commutative algebra and in particular the Cartan calculus \eqref{eqn:CartanCalc}, we see that the Lie derivative $L_X$ for any $X\in \rho\Der(\A)$ is an infinitesimal symmetry.
\end{example}

\section*{Acknowledgements}
We cordially thank Prof.~Tomasz~Brzezi\'{n}ski and Prof.~Richard~Szabo for their interest and valuable advice.  We furthermore, thank the anonymous referee for their invaluable comments and suggestions. 

\end{document}